\documentclass[12pt]{article}
\usepackage{amsmath,amssymb}
\usepackage[latin2]{inputenc}
\usepackage{amsthm}
\usepackage{amssymb}
\usepackage[T1]{fontenc}
\usepackage{amsmath}
\usepackage{amsbsy}
\usepackage{graphicx}
\newtheorem{theorem}{Theorem}[section]
\newtheorem{corollary}{Corollary}[section]

\def\proclaim#1{\par \bigskip\noindent {\bf #1}\bgroup\it\ }
\def\endproclaim{\egroup\par\bigskip}

\newbox\TempBox \newbox\TempBoxA

\newfam\openfam

\def\ci{\perp\!\!\!\perp}

\def\text#1{\mbox{\rm #1}}
\def\overset#1#2{\stackrel{#1}{#2} }

\def\underwiggle 1{
\ifmmode\setbox\TempBox=\hbox{$ 1$}\else\setbox\TempBox=\hbox{ 1}\fi
\setbox\TempBoxA=\hbox to \wd\TempBox{\hss\char'176\hss}
\rlap{\copy\TempBox}\smash{\lower9pt\hbox{\copy\TempBoxA}} }

\allowdisplaybreaks
\newtheorem{lem}{Lemma}

\theoremstyle{remark}

\begin{document}

\title{\bf  On the strong approximations of partial sums of $ f(n_kx)$}
\author{Marko Raseta$^{\rm 1}$ }
\date{}
 \maketitle

\begin{abstract}
We prove a strong invariance principle for the sums $\sum_{k=1}^N f(n_kx)$,
where $f$ is a smooth periodic function on $\mathbb R$ and $(n_k)_{k\ge 1}$ is an increasing
random sequence. Our results show that in contrast to the classical Salem-Zygmund theory,
the asymptotic properties of lacunary series with random gaps can be described very precisely
without any assumption on the size of the gaps.
\end{abstract}

\renewcommand{\thefootnote}{}

\footnote{ {\it 2010 Mathematics Subject Classification.}  Primary 60F17, 42A55, 42A61.}

\footnote{ {\it Keywords.} Law of the iterated logarithm, lacunary series, random indices.}

\footnote{$^{\rm 1)}$Department of Mathematics, University of York Email: {\tt
marko.raseta@york.ac.uk}}

\section{Introduction}

Let $f: \mathbb{R}\to \mathbb{R}$ be a measurable function satisfying
\begin{equation}\label{fcond}
f(x+1)=f(x), \quad \int_0^1 f(x) dx=0, \quad \|f\|^2= \int_0^1 f^2(x) dx <\infty.
\end{equation}
It is well known that for rapidly increasing $(n_k)_{k\ge 1}$ the sequence
$(f(n_kx))_{k\ge 1}$ behaves like a sequence of independent random variables.
For example, if
\begin{equation}\label{infgap}
n_{k+1}/n_k\to\infty
\end{equation}
and $f$ is a Lipschitz function, then
\begin{equation}\label{clt}
N^{-1/2} \sum_{k=1}^N f(n_kx) \overset{d}{\longrightarrow} N(0, \|f\|^2)
\end{equation}
and
\begin{equation}\label{lil}
\limsup_{N\to\infty} (2N\log\log N)^{-1/2} \sum_{k=1}^N f(n_kx) =\|f\| \qquad \
\text{a.s.}
\end{equation}
with respect to the probability space $(0, 1)$ equipped with Borel sets and Lebesgue measure
(see Takahashi \cite{tak1}, \cite{tak2}). Assuming only the Hadamard gap condition
\begin{equation}\label{had}
n_{k+1}/n_k \ge q>1, \qquad k=1, 2, \ldots
\end{equation}
the situation becomes more complicated. Kac \cite{ka} proved that $f(n_kx)$ satisfies the CLT for $n_k=2^k$ and
Erd\H{o}s and Fortet (see also \cite{ka}) showed that this generally fails for $n_k=2^k-1$. Gaposhkin \cite{gap} showed that
$f(n_kx)$ satisfies the CLT provided the ratios $n_{k+1}/n_k$ are integers or $n_{k+1}/n_k\to \alpha>1$ where
$\alpha^r$ is irrational for $r=1, 2, \ldots$. A necessary and sufficient number-theoretic condition for the CLT
for $f(n_kx)$ under  (\ref{had}) was given by Aistleitner and Berkes \cite{aibe}.

For sequences $(n_k)_{k\ge 1}$ growing slower than exponentially, the asymptotic behavior of $S_N= \sum_{k=1}^N f(n_kx)$
still depends on arithmetic properties of $n_k$, but the arising number theoretical problems become essentially
intractable. As a consequence, the asymptotic distribution (if it exists) of normed sums of $f(n_kx)$ is not known even
for $f(x)=\sin x$ and
simple sequences like $n_k=k^r$ $(r=2, 3, \ldots)$. On the other hand, as it happens frequently in analysis, the "typical"
behavior of the sum $S_N$ is much more manageable than individual cases. For example, it is an open question if $\sin n_kx$
satisfies the central limit theorem for $n_k=e^{k^\alpha}$,  and all $\alpha>0$  (for more information, see Erd\H{o}s
\cite{er62}); on the other hand, Kaufman \cite{kau} showed that the CLT holds for $n_k=e^{ck^\alpha}$ for all $\alpha>0$
and almost all $c>0$ in the sense of Lebesgue measure.
Random constructions have also been used to solve many problems in harmonic analysis, see e.g.\ Salem and Zygmund
\cite{sz1954}, Erd\H{o}s \cite{er54}, Halberstam and Roth \cite{hr}, Berkes \cite{be78}, \cite{be79}, Bobkov and G\"otze 
\cite{bogo},
Fukuyama \cite{fu1},
\cite{fu2}, \cite{fu3}, Aistleitner and Fukuyama \cite{af}. The purpose of this paper
is to prove a strong invariance principle for $\sum_{k=1}^N f(n_kx)$ in the case when $(n_k)_{k\ge 1}$ is an increasing
random walk, i.e.\ $n_{k+1}-n_k$ are i.i.d.\ positive random variables. This model was introduced by Schatte \cite{sch2}
and it has remarkable properties, see Schatte \cite{sch2}, \cite{sch3}, Weber \cite{we}, Berkes and Weber \cite{bewe}.
In this paper we will prove the following result.

\bigskip\noindent
{\bf Theorem.} {\it Let $(X_n)_{n \in \mathbb N}$ be a sequence of i.i.d.\ positive
random variables defined on a sufficiently large probability space $(\Omega, \mathcal{F}, \mathbb{P})$
and let $S_n=\sum_{k=1}^n X_k$. Assume $X_1$ is bounded with bounded density.
Let $f$ be a \rm{Lip} ($\alpha$) function  satisfying (\ref{fcond}) and put
\begin{align}\label{axdef}
A_x = \|f\|^2 + 2 \sum\limits^\infty_{k =1} \mathbb E f(U)
f(U + S_k x),
\end{align}
where $U$ is a uniform $(0, 1)$ random variable, independent of $(X_j)_{j \ge 1}$. Then there exists a
Brownian motion such that
\[
\sum\limits_{k = 1}^n f(S_k x) = W(A_x n) + O\left(n^{5/12 + \varepsilon} \log n\right) \qquad \text{a.s.}
\]
for all $\varepsilon > 0$.}

\bigskip
Clearly, the existence of such an $U$ can always be guaranteed by an enlargement of the probability space. The absolute
convergence of the series in (\ref{axdef}) will follow from the proof of our
theorem.

Changing the notation slightly, our theorem and Fubini's theorem imply immediately the LIL
\begin{equation}\label{lil2}
\limsup_{N\to\infty}\, (2N\log\log N)^{-1/2} \sum_{k=1}^N f(n_kx) =A_x^{1/2} \qquad \text{a.s.}
\end{equation}
and its functional version for almost all sequences $(n_k)$ generated by the random walk model.
Note that, in contrast to the nonrandom case (\ref{lil}), the limsup in (\ref{lil2}) is a function of $x$.
Other immediate consequences are Chung's lower LIL
\begin{equation}\label{lil3}
\liminf_{N\to\infty} \, \left(\frac{\log\log N}{N}\right)^{1/2}  \max_{1\le M\le N} \left| \sum_{k=1}^M f(n_kx)\right|
=\frac{8}{\pi^2} A_x^{1/2} \qquad \text{a.s.}
\end{equation}
and the Kolmogorov-Erd\H{o}s-Feller-Petrovski test stating that for any positive nondecreasing function $\varphi$ on $(0, 
\infty)$
and almost all $x$ the inequality
$$ \sum_{k=1}^M f(n_kx) > \sqrt{N} \varphi(N)$$
holds for finitely or infinitely many $N$ according as the integral
$$\int_1^\infty \frac{\varphi(t)}{t} e^{-\varphi(t)^2/2} dt$$
converges or diverges.

\bigskip

\section{\bf Proof of the theorem}
We start with some preparatory results. All sums here are considered modulo 1.
\begin{lem}[Schatte, \cite{sch2}]
Let the three random vectors $X=(X_1, \dots, X_r)$, $U$, and $(W_1,\dots ,W_s)$
be independent and let $W=f(X)$ with a measurable function $f$. If $U$ is uniformly
distributed, then the two random vectors $X$ and $(W+ U + W_1, W+ U
+ W_2, \dots , W+ U + W_s)$ are independent, too.
\end{lem}
\begin{lem}[Schatte, \cite{sch2}]
Let $W$ and $U$ be independent random variables, where $U$ is uniformly
distributed. Then $W+ U$ is independent of $W$.
\end{lem}
\begin{lem}[Schatte, \cite{sch2}]
Let $X$ be a random variable with the distribution function $F(x)$, where
$\sup_{0\leq x<1} |F(x)-x|\leq\varepsilon$. Further let $U$ be a uniformly distributed random variable being independent of 
$X$. Then there exists a uniformly distributed random variable $V$ such that
\begin{enumerate}
\item[(i)] $|V-X|\leq\varepsilon$
\item[(ii)] $V=f(U,X)$, where $f$ is measurable.
\end{enumerate}
\end{lem}
\begin{theorem}[Schatte, \cite{sch1}]
Let $p_n(x)$ denote the density of $Y_n=\sum_{i=1}^nX_i$, where $X_i$s are i.i.d. random variables. Then the following
assertions are equivalent;
\begin{enumerate}
\item[a)]The density $p_m(x)$ is bounded for some $m$.
\item[b)]$\sup_{0\leq x<1}|p_n(x)-1|\rightarrow 0$ as $n\rightarrow \infty$.
\item[c)]$\sup_{0\leq x<1}|p_n(x)-1|\leq Cw^n$, where $C$ and $w<1$ are real constants.
\end{enumerate}
\end{theorem}
Last four results and simple induction yield the following:
\begin{lem}\label{lem:1}
Fix $l\in {\mathbb N}$, $x\ne 0$ and define a sequence of sets by
\begin{align*}
I_1 &:= \{1,2,\dots, \beta\}\\
I_2 &:= \{p_1, p_1 + 1, \dots, p_1 + \beta_1\}\ \text{ where } \ p_1 \geq \beta + l + 2\\
\vdots & \\
I_n &:= \left\{ p_{n - 1}, p_{n - 1} + 1, \dots, p_{n - 1} +
\beta_{n - 1}\right\}\
\text{ where }\ p_{n - 1} \geq p_{n - 2} + \beta_{n - 2} + l + 2\\
\vdots &
\end{align*}
Then there exists a sequence $\delta_1^x, \delta_2^x, \dots$ of random variables satisfying the following properties:

\begin{itemize}
\item[\rm (i)] $|\delta_n^x| \leq C_x e^{-\lambda_x l}$ for all $n \in \mathbb N$, where
$\lambda_x$ and $C_x$ are some positive quantities that depend on $x$ only.

\item[\rm (ii)] The random variables
\[
\sum_{i \in I_1} f(S_i x), \ \sum_{i \in I_2} f(S_i x - \delta_1^x),
\dots, \ \sum_{i \in I_n} f(S_i x - \delta_{n - 1}^x), \dots
\]
are independent.
\end{itemize}
\end{lem}

In an identical fashion to Lemma \ref{lem:1} we can prove the following result.
\begin{lem}\label{lem:1:1}
Let $p,q,c,s$ be 4 natural numbers listed in an increasing order. Then there exists a random variable $T$ with 
$T=O(e^{-\lambda(q-p)})$ for some constant $\lambda$ such that
\begin{enumerate}
\item[(i)] $(S_q-T, S_r-T, S_s-T)$ is a random vector with all three components having uniform distribution.
\item[(ii)] Random vector in (i) is independent of $S_p$.
\end{enumerate}
\end{lem}
\begin{corollary}\label{cor:1}
Fix $x\in \mathbb R$. Define a sequence of random variables on $(\Omega,\mathcal F, \mathbb P)$ by
$$
Y_k^x=f(S_kx)-\mathbb E F(S_kx).
$$
Then for any sequence $(a_k)_{k\in \mathbb N}$ we have
$$
\mathbb E\left( \sum_{k=1}^na_kY_k^x \right)^4\leq C_x\left( \sum_{k=1}^na_k^2 \right)^2
$$
for some constant $C_x$ depending on $x$ only.
\end{corollary}
\begin{proof}
As a direct consequence of Lemma \ref{lem:1:1} we can find a random variable $T^x$ such that
$$
(S_qx-T^x, S_rx-T^x, S_sx-T^x)\ci S_px.
$$
For simplicity define
$$
Z_{q,r,s}^x:=(f(S_qx-T^x)-\mathbb E f(S_qx))( f(S_rx-T^x)-\mathbb E f(S_rx))( f(S_sx-T^x)-\mathbb E f(S_sx))
$$
Then it is clear that $Y_p^x$ is independent of $Z_{q,r,s}^x$ and that they are both zero mean random variables. Using Lemma 
\ref{lem:1:1} one can see that
$$
|\mathbb EY_p^xY_q^xY_r^xY_s^x|\leq C_xe^{-D_x(q+s-p-r)}.
$$
But then it follows that
$$
\mathbb E\left( \sum_{k=1}^n a_kY_k^x \right)\leq C_x\sum_{1\leq p\leq q\leq r\leq s\leq 
n}|a_p||a_q||a_r||a_s|e^{-D_x(|q-p|+|s-r|)}
$$
Now let $q-p=\omega$, $s-r=\zeta$. We have
\begin{align*}
&\sum_{1\leq p\leq q\leq r\leq s\leq n}|a_p||a_q||a_r||a_s|e^{-(|q-p|+|s-r|)}=\\
&\sum_{1\leq p, p+\omega, r, r+\zeta\leq n}|a_p||a_{p+\omega}||a_r||a_{r+\zeta}|e^{-\omega+\zeta}\leq\\
&e^{-\omega+\zeta}\sum_{1\leq p, p+\omega\leq n}|a_p||a_{p+\omega}|\sum_{1\leq r, r+\zeta\leq n}|a_r||a_{r+\zeta}|\leq\\
&e^{-\omega+\zeta}\left(\sum_{1\leq p\leq n}a_p^2\right)^{1/2}\left(\sum_{1\leq p+\omega\leq n}a_{p+\omega}^2\right)^{1/2}
\left(\sum_{1\leq r\leq n}a_r^2\right)^{1/2}\left(\sum_{1\leq r+\zeta\leq n}a_{r+\zeta}^2\right)^{1/2}\leq\\
&e^{-\omega+\zeta}\left( \sum_{1\leq k\leq n} a_k^2 \right)^{1/2},
\end{align*}
and the proof is now complete.
\end{proof}

\medskip
Put  $\widetilde m_k = \sum\limits^k_{j = 1} \lfloor
j^{1/2}\rfloor$, $\widehat m_k = \sum\limits^k_{j = 1} \lfloor
j^{1/4}\rfloor$ and let $m_k = \widetilde m_k + \widehat m_k$.
Using Lemma 4 we can construct sequences $(\Delta_k^x)_{k \in \mathbb N}$,
$(\Pi_k^x)_{k \in \mathbb N}$ of random variables such that setting
\begin{align*}
T_k &:= \sum^{m_{k - 1} + \lfloor \sqrt{k}\rfloor}_{j = m_{k - 1} +
1} \left(f(S_j x - \Delta_{k - 1}^x) -
\mathbb E f(S_j x - \Delta_{k - 1})\right)\\
T_k^* &:= \sum^{m_k}_{j = m_{k - 1} + \lfloor\sqrt{k}\rfloor + 1}
\left(f(S_j x - \Pi_{k - 1}^x) - \mathbb E f(S_j x - \Pi_{k -
1}^x\right)
\end{align*}
we have

\medskip
{\rm (i)} $\Delta_0^x = 0$; $|\Delta_k^x| \leq C_x e^{-\lambda_x
\sqrt[4]{k}}$;  $(T_k)_{k \in \mathbb N}$ is a
sequence of independent random variables,

{\rm (ii)} $\Pi_0^x = 0$; $|\Pi_k^x| \leq C_x e^{-\lambda_x
\sqrt{k}}$;  $(T_k^*)_{k \in \mathbb N}$ is a sequence of
independent random variables.

\medskip\noindent
Routine stationary-type arguments yield the following:
\begin{lem} \label{lem:2}
\begin{equation*}
\sum^n_{k = 1} \mathbb V\text{\rm ar} (T_k) \sim A_x \widetilde m_n,
\qquad  \sum^n_{k = 1} \mathbb V\text{\rm ar} (T_k^*) \sim A_x
\widehat m_n.
\end{equation*}
where $A_x$ is defined by (\ref{axdef}).
\end{lem}

The following lemma is a special case of Strassen's strong approximation theorem.

\begin{lem}\label{Strassen}
Let $Y_1, Y_2, \dots$ be independent r.v.'s with finite fourth moments, let
$a_n = \sum_{i = 1}^n \mathbb E Y_i^2$ and assume
$$
\sum_{n = 1}^\infty \mathbb E Y_n^4 / a_n^{2\vartheta} < \infty
$$
with $0 < \vartheta < 1$.
Then the sequence $Y_1, Y_2, \dots$ can be redefined on a new probability space together with a Wiener process $\zeta (t)$ 
such that
\[
Y_1 + \dots + Y_n = \zeta(a_n) + o \left({a_n}^{(1 + \vartheta)/4} \log a_n\right) \ \text{ a.s.}
\]
\end{lem}

\bigskip\noindent
We are now ready to prove our result.

\medskip\noindent
As before, let
\[
T_k = \sum\limits_{j = m_{k - 1} + 1}^{m_{k - 1} + \lfloor\sqrt{k}\rfloor} f\left(S_j x - {\Delta_{k - 1}}^x\right)
\ \text{ and } \
{T_k}^* = \sum\limits_{j = m_{k - 1} + \lfloor \sqrt{k}\rfloor + 1}^{m_k} f \left(S_j x - {\Pi_{k - 1}}^x\right).
\]
We will apply Lemma \ref{Strassen} for both $(T_k)_{k \in \mathbb N}$ and $({T_k}^*)_{k \in \mathbb N}$.
Clearly, $(T_k)_{k \in \mathbb N}$ is a sequence of independent, zero-mean random variables and $|T_k| \leq M \sqrt{k}$.
Corollary 2.1 gives as:
\[
\mathbb E \left(\sum_{k = 1}^N c_k Y_k^x\right)^4 \leq C_x \cdot \left(\sum_{k = 1}^N {c_k}^2\right)^2
\]
where $Y_k^x = f(S_kx) - \mathbb E f(S_kx)$.
But then
\[
\mathbb E {T_k}^4 \leq C_1^x k + C_2^x k^2 e^{-\alpha \lambda_x \sqrt[4]{k - 1}}
\]
for some suitable constants $C_1^x$ and $C_2^x$.
Now observe that
\[
V_n = \sum_{k = 1}^n \mathbb E Y_k^2 \sim A_x m_n \sim A_x \cdot \frac23 n^{3/2}
\]
by Lemma~\ref{lem:2}.
These facts together imply that
$$\sum\limits_{k \in \mathbb N} \frac{\mathbb E {T_k}^4}{m_k^{2(2/3 + \varepsilon)}} < \infty$$
for all $\varepsilon > 0$ and thus by Lemma \ref{Strassen} we get
\[
\sum_{k = 1}^n T_k = \zeta(A_x m_n) + o \left(m_n^{\left(1 + \frac23 + \varepsilon\right)/4} \log m_n\right) \quad 
\text{a.s.}
\]
Define a sequence $(p(n))_{n \in \mathbb N}$ of integers by
\[
m_{p(n)} \leq n < m_{p(n) + 1}.
\]
Clearly,
\[
\sum_{k = 1}^{p(n)} T_k = \zeta \left(A_x m_{p(n)}\right) + o \left({m_{p(n)}}^{\left(1 + \frac23 + \varepsilon\right)/4} 
\log m_{p(n)}\right);
\quad \text{a.s.}
\]
and similarly
\[
\sum_{k = 1}^{p(n)} {T_k}^* = \zeta' \left(A_x {\widehat m_{ p(n)}}\right) + o \left({{\widehat m_{ p(n)}}}^{\left(1 + 
\frac35 + \delta\right)/4} \log m_{p(n)}\right); \quad \text{a.s.}
\]
for some other Brownian motion $\zeta'$.
Now
\begin{align*}
\sum_{k = 1}^n f(S_k x) &= \sum_{k = 1}^{p(n)} T_k + \sum_{k = 1}^{p(n)} {T_k}^* + \sum_{k = 1}^{p(n)} \sum_{j = m_{k - 1} + 
1}^{m_{k - 1} + \lfloor \sqrt{k}\rfloor} \left(f(S_j x) - f(S_j x - {\Delta_{k - 1}}^x)\right)\\
&\quad + \sum_{k = 1}^{p(n)} \sum_{j = m_{k - 1} + \lfloor \sqrt{k}\rfloor + 1}^{m_k} \left(f(S_j x) - f(S_j x - {\Pi_{k - 
1}}^x)\right) + \sum_{k = m_{p(n)} + 1}^n f(S_k x).
\end{align*}
Thus
\begin{align*}
\sum_{k = 1}^n f(S_k x)\! -\! \zeta (A_x n) &= \sum_{k = 1}^{p(n)} T_k - \zeta \left(A_x m_{p(n)}\right) + \zeta\left(A_x 
m_{p(n)}\right) - \zeta (A_x n) \\
&\ + \sum_{k = 1}^{p(n)} {T_k}^* - \zeta'\left(A_x {\widehat m_{ p(n)}}\right) + \zeta'\left(A_x {\widehat m_{ p(n)}}\right) 
\\
&\ + \sum_{k = 1}^{p(n)} \sum_{j = m_{k - 1} + 1}^{m_{k - 1} + \lfloor \sqrt{k}\rfloor} \left(f(S_j x) - f(S_j x - {\Delta_{x 
- 1}}^x)\right) \\
&\ + \sum_{k = 1}^{p(n)} \sum_{j = m_{k - 1} + \lfloor \sqrt{k}\rfloor + 1}^{m_k}\! \left(f(S_j x) \!-\! f(S_j x \! -\! 
{\Pi_{k - 1}}^x)\right) + \! \sum_{k = m_{p(n) + 1}}^n\! f(S_k x).
\end{align*}
We investigate each term separately.

\medskip\noindent
(i) We have seen that
\[
\sum_{k = 1}^{p(n)} T_k - \zeta (A_x m_{p(n)}) = o\left({m_{p(n)}}^{(1 + 2/3 + \varepsilon)/4} \log m_{p(n)}\right).
\]
But $m_{p(n)} \sim n$ and thus
\[
\sum_{k = 1}^{p(n)} T_k - \zeta (A_x m_{p(n)}) = o \left( n^{\frac{5}{12} + \frac{\varepsilon}{4}} \log n\right) = o 
\left(n^{\frac{5}{12} + \varepsilon} \log n\right).
\]

\medskip\noindent
(ii) We have
\begin{align*}
\left|\zeta(A_x m_{p(n)}) - \zeta (A_x n)\right| &= \left| \zeta(A_x n) - \zeta(A_x m_{p(n)})\right| \\
&\overset{d}{=} \left| N(0, A_x (n - m_{p(n)})\right|
\end{align*}
Now
$$n - m_{p(n)} \leq \lfloor(p(n) + 1)^{1/2}\rfloor + \lfloor (p(n) + 1)^{1/4}\rfloor \leq C n^{1/3}$$
for all $n$ large enough.
Thus
$$\sum\limits_{n \in\mathbb N} \mathbb P \left(\left| \zeta (A_x m_{p(n)}) - \zeta (A_x n) \right| \geq n^{7/24}\right) < 
\infty$$
and consequently
\[
\left|\zeta(A_x m_{p(n)}) - \zeta(A_x n)\right| \leq n^{7/24} \quad \text{a.s.}
\]
for all $n$ large enough.
Hence
\[
\frac1{n^{5/12 + \varepsilon} \log n} \left|\zeta(A_x m_{p(n)}) - \zeta (A_xn)\right| \to 0 \qquad \text{a.s.}
\]

\medskip\noindent
(iii) We have also seen that
\[
\sum_{k = 1}^{p(n)} {T_k}^* - \zeta'(A_x {\widehat m_{ p(n)}}) = o\left({{\widehat m_{ p(n)}}}^{(1 + 3/5 + \delta)/4} \log
{\widehat m_{ p(n)}}\right).
\]
But $\hat m_n \sim Cn^{5/4}$ and thus  ${\widehat m_{ p(n)}} \sim Cn^{5/6}$.
Hence for sufficiently small $\delta$ we have
\[
\frac{1}{n^{5/12 + \varepsilon} \log n} \left|\sum_{k = 1}^{P(n)} {T_k}^* - \zeta'\left(A_x {\widehat m_{ p(n)}}\right) 
\right|
\leq \frac{C n^\frac{8 + 5\delta}{24}}{n^{5/12 + \varepsilon} \log n} \to 0 \quad \text{a.s.}
\]
as $n\to\infty$.

\medskip\noindent
(iv) Clearly $\zeta'(A_x {\widehat m_{ p(n)}}) \overset{d}{=} \sqrt{A_x \, {\widehat m_{ p(n)}}} \, N(0,1)$.
Observe that $\widehat{m}_n \sim C n^{5/4}$, $p(n) \sim C n^{2/3}$ and thus ${\widehat m_{ p(n)}} \sim Cn^{5/6}$.
Consequently,
$$ \dfrac{\left|\zeta'(A_x {\widehat m_{ p(n)}})\right|}{n^{5/12 + \varepsilon} \log n} \leq \dfrac{n^{5/12 + 
\varepsilon/2}}{n^{5/12 + \varepsilon} \log n} \to 0 \qquad \text{a.s.} $$
as $n \to \infty$.

\medskip\noindent
(v)
We have
\begin{align*}
&\left|\sum_{k = 1}^{p(n)} \sum_{j = m_{k - 1} + 1}^{m_{k - 1} + \lfloor \sqrt{k}\rfloor} \left(f(S_j x) - f(S_j x - 
{\Delta_{k - 1}}^x)\right)\right| \\
&\leq \sum_{k = 1}^{p(n)} \sum_{j = m_{k - 1} + 1}^{m_{k - 1} + \lfloor \sqrt{k}\rfloor} \left| f(S_j x) - f(S_j x - 
{\Delta_{k - 1}}^x)\right| \\
&\leq \sum_{k = 1}^{p(n)} \sum_{j = m_{k - 1} + 1}^{m_{k - 1} + \lfloor \sqrt{k}\rfloor} C_x^\alpha e^{-\alpha \lambda x 
\lfloor \sqrt[n]{k - 1}\rfloor} =
\sum_{k = 1}^{p(n)} C_x^\alpha \lfloor \sqrt{k}\rfloor e^{-\alpha \lambda_x \lfloor \sqrt[n]{k - 1}\rfloor}\\
&\leq \sum_{k = 1}^\infty C_x^\alpha \lfloor \sqrt{k}\rfloor e^{-\alpha \lambda_x \lfloor \sqrt[\eta]{k - 1}\rfloor} < 
\infty.
\end{align*}
Thus
$$\dfrac{1}{n^{5/12 + \varepsilon} \log n} \sum\limits_{k = 1}^{p(n)} \sum\limits_{j = m_{k - 1} + 1}^{m_{k - 1} + \lfloor 
\sqrt{k}\rfloor}
\left(f(S_j x) - f(S_j x - {\Delta_{k - 1}}^x)\right) \to 0 \qquad \text{a.s.} $$
as $n \to \infty$.

\medskip\noindent
(vi) 
Arguing exactly as in (v), we get
\[
\dfrac{1}{n^{5/12 + \varepsilon} \log n} \sum\limits_{k = 1}^{P(n)} \sum\limits_{j = m_{k - 1} + \lfloor \sqrt{k}\rfloor + 
1}^{m_k}
\left(f(S_j x) - f(S_j x - {\Pi_{k - 1}}^x)\right) \to 0\ \qquad \text{a.s.}
\]

\medskip\noindent
(vii)
We have
\begin{align*}
\left|\sum_{k = m_{p(n)} + 1}^n \!\! f(S_k x)\right| &\leq \sum_{k = m_{p(n)} + 1}^n \left|f(S_k x)\right| \leq M(n - 
m_{p(n)}) \\
&\leq M\!\left(\!\left\lfloor (p(n) + 1)^{1/2} \right\rfloor
\! +\!  \left\lfloor (p(n) + 1)^{1/4}\right\rfloor\! \right)\! \leq 2M \!\left\lfloor (p(n)\! +\! 1)^{1/2}\right\rfloor \! 
\sim 2
M n^{1/3}.
\end{align*}
Thus
$$\dfrac1{n^{5/12 + \varepsilon} \log n} \sum\limits_{k = m_{p(n)} + 1}^n f(S_k x) \to 0 \qquad \text{a.s.} $$
as $n \to \infty$. Summarizing the above estimates, we obtain our result.

\end{document}